\documentclass[11pt]{article}

\usepackage{fullpage}
\usepackage{mathreport}
\usepackage{authblk}
\usepackage{multirow}
\usepackage{tikz}
\usepackage{circuitikz}
\usepackage{biblatex}
\usetikzlibrary{arrows}

\DeclareMathOperator{\Tr}{Tr}

\newcommand{\lp}[1]{P_{#1}^{\gamma}}
\newcommand{\Star}[1]{\widehat{K}_{1,#1}}

\title{State Transfer in Complex Quantum Walks}
\ignore{ 
\author{
Antonio Acuaviva\\{Universidad Complutense de Madrid}
\and 
Ada Chan\\{York University}
\and 
Summer Eldridge\\{University of Toronto}
\and 
Chris Godsil\\{University of Waterloo}
\and 
Matthew How-Chun-Lun\\{McMaster University}
\and 
Christino Tamon\\{Clarkson University}
\and 
Emily Wright\\{Queen's University}
\and 
Xiaohong Zhang\\{University of Waterloo}
}
}

\author[1]{Antonio Acuaviva}
\author[2]{Ada Chan}
\author[3]{Summer Eldridge}
\author[4]{Chris Godsil}
\author[5]{Matthew How-Chun-Lun}
\author[6]{Christino Tamon}
\author[7]{Emily Wright}
\author[8]{Xiaohong Zhang}

\affil[1]{Department of Mathematics, Universidad Complutense de Madrid}
\affil[2]{Department of Mathematics and Statistics, York University}
\affil[3]{Department of Mathematics, University of Toronto}
\affil[4]{Department of Combinatorics and Optimization, University of Waterloo}
\affil[5]{Department of Mathematics, McMaster University}
\affil[6]{Department of Computer Science, Clarkson University}
\affil[7]{Department of Mathematics, Queen's University}
\affil[8]{Centre de recherches math\'{e}matiques, Universit\'{e} de Montr\'{e}al}

\date{\today}

\begin{document}
\maketitle

\begin{abstract}
Given a graph with Hermitian adjacency matrix $H$, perfect state transfer occurs from vertex $a$ to vertex $b$ 
if the $(b,a)$-entry of the unitary matrix $\exp(-iHt)$ has unit magnitude for some time $t$.
This phenomenon is relevant for information transmission in quantum spin networks
and is known to be monogamous under real symmetric matrices.
We prove the following results:
\begin{itemize}
\item For oriented graphs (whose nonzero weights are $\pm i$), 
    the oriented $3$-cycle and the oriented edge are the only graphs where perfect state transfer occurs between every pair of vertices.
    This settles a conjecture of Cameron \etal \cite{USTG}.
    On the other hand, we construct an infinite family of oriented graphs 
    with perfect state transfer between any pair of vertices on a subset of size four.
\item There are infinite families of Hermitian graphs with one-way perfect state transfer,
    where perfect state transfer occurs without periodicity. 
    In contrast, perfect state transfer implies periodicity whenever the adjacency matrix has algebraic entries 
    (see Godsil \cite{RST}).
\item There are infinite families with non-monogamous pretty good state transfer in rooted graph products.
    In particular, we generalize known results on double stars (due to Fan and Godsil \cite{Fan_2012})
    and on paths with loops (due to Kempton, Lippner and Yau \cite{kly17}).
    The latter extends the experimental observation of quantum transport (made by Zimbor\'{a}s \etal \cite{zfkwlb})
    and shows non-monogamous pretty good state transfer can occur amongst distant vertices.
\end{itemize}
\end{abstract}

\section{Introduction}

Given a graph $X=(V,E)$ with adjacency matrix $A$, a continuous-time quantum walk on $X$ is defined
by the time-dependent unitary matrix $U(t) = e^{-iAt}$. This natural quantum generalization of
continuous-time random walks is important for designing quantum algorithms. Childs \etal \cite{Childs_2003}
showed that a continuous-time quantum walk algorithm provides an exponential time speedup for an
explicit search problem on graphs.
Subsequently, Childs \cite{c09} showed that continuous-time quantum walk is a universal model of
quantum computation.

Our focus in this paper is motivated by Bose \cite{b03} who studied quantum communication via
continuous-time quantum walk on graphs. We say that there is {\em pretty good state transfer} in a graph $X$
from vertex $a$ to vertex $b$ if for any $\epsilon > 0$, there is a time $t$ so that
$\norm{U(t)e_a - \gamma e_b} \le \epsilon$ where $\gamma$ is a phase factor.
Here, $e_a$ denotes the unit vector with $1$ at position $a$ and $0$ elsewhere; similarly for $e_b$.
If $\epsilon = 0$ is achievable, we say there is {\em perfect} state transfer in $X$ from $a$ to $b$
at time $t$.

Kay \cite{k11} proved a monogamy property for perfect state transfer on graphs with real symmetric
adjacency matrices: if there is perfect state transfer from $a$ to $b$ and from $a$ to $c$ then $b=c$.
In contrast, Cameron \etal \cite{USTG} showed that there are {\em oriented} graphs
(whose adjacency matrices are Hermitian with $\pm i$ nonzero entries) where state transfer occurs between
every pair of vertices. This latter property is called {\em universal} state transfer.
Their primary examples are oriented cycles of prime order with universal pretty good
state transfer. A notable exception is the oriented $3$-cycle which exhibits universal
{\em perfect} state transfer.

It was conjectured in \cite{USTG} that the oriented $K_2$ and $3$-cycle are the only oriented
graph with universal perfect state transfer. We prove their conjecture in this work. This confirms 
that universal perfect state transfer is an extremely rare phenomenon in oriented graphs.
On the other hand, there are known infinite families of graphs with universal perfect state transfer
but with adjacency matrices that are Hermitian matrices with no restriction on the entries
(see Connelly \etal \cite{UPST}).  We call these Hermitian graphs.

Godsil and Lato \cite{gl} proved a strong characterization of perfect state tranfer in oriented graphs 
and observed that perfect state transfer always implies periodicity (by the Gelfond-Schneider theorem).
In fact, Godsil \cite{RST} had observed the latter property holds for any adjacency matrix
with algebraic entries. 
Our next observation shows that the latter assumption is necessary to guarantee periodicity.
We construct the first infinite family of Hermitian graphs with {\em one-way} perfect state transfer,
where perfect state transfer occurs {\em without} periodicity.
These examples also exhibit a {\em one-time} perfect state transfer property
where perfect state transfer occurs at a single unique time (and never to repeat again).

Godsil and Lato \cite{gl} also introduced a relaxation of universal perfect state transfer
called {\em multiple perfect state transfer}.
We say a graph $X$ has {\em multiple} state transfer on a subset $S \subset V(X)$
of vertices, with $|S| \geq 3$, if state transfer occurs between every pair of vertices of $S$.
An explicit example of a $8$-vertex circulant with multiple perfect state tranfer was given in \cite{gl},
but it was not clear if there are more examples sharing the same properties.
We construct the first infinite family of oriented graphs with multiple perfect state transfer
(which contains the aforementioned $8$-vertex circulant as a special case).
This shows that, unlike universal perfect state transfer, multiple perfect state transfer
is not an extremely rare phenomenon.

It is known that perfect state transfer is closed under the Cartesian graph product.
In this work, under mild assumptions, we show that multiple state transfer 
is closed under the rooted graph product (see Godsil and McKay \cite{McKay_1978}).
First, we prove a complete characterization of pretty good state transfer on the rooted product
of the oriented $3$-cycle with stars $K_{1,m}$.
This generalizes a result of Fan and Godsil \cite{Fan_2012} on the double stars.
Next, we consider rooted product with single-looped paths instead of stars.
Let $X$ be a $n$-vertex circulant with universal perfect state transfer and let $P_m^\gamma$ be
a $m$-vertex path with a self-loop of weight $\gamma$ at one of its endpoints.
We prove that the rooted product $X \circ P_m^\gamma$ has multiple pretty good state transfer between
every pair of vertices with self-loop provided $\gamma$ is transcendental.
This generalizes a result of Kempton, Lippner and Yau \cite{kly17} and shows the power of loops
to facilitate multiple state transfer among distant vertices. 
In the special case when $X$ is the oriented $3$-cycle, our result strengthens the
experimental observations in Zimbor\'{a}s \etal \cite{zfkwlb} (with the help of self-loops).

\section{Preliminary}

Given a graph $X$ and an associated Hermitian matrix $H$, the transition matrix of its continuous-time quantum walk is 
\begin{equation*}
U(t) = e^{-\ii t H}.
\end{equation*}
We call $X$ a {\em Hermitian graph} if we do not assume any additional condition on the entries of $H$.
For the special case where $X$ is an oriented graph, we use the Hermitian matrix $H$ defined as
\begin{equation*}
H_{a,b} = 
\begin{cases}
\ii & \text{if there is an arc from $a$ to $b$ in $X$,}\\
-\ii & \text{if there is an arc from $b$ to $a$ in $X$, and}\\
0 & \text{if there is no arc between $a$ and $b$ in $X$.}\\
\end{cases}
\end{equation*}

Let $\theta_1, \ldots, \theta_d$ be distinct eigenvalues of $H$.
For $r=1,\ldots, d$, let $E_r$ denote the orthogonal projection matrix onto the $\theta_r$-eigenspace of $H$.  
Then $E_r E_s = \delta_{r,s} E_r$ and $\sum_r E_r=I$.  The spectral decomposition $H=\sum_{r} \theta_r E_r$ gives
\begin{equation*}
U(t) = \sum_{r=1}^d e^{-\ii t \theta_r}E_r.
\end{equation*}
Given a unit vector $v\in \mathbb{C}^n$,  the system with initial state $v$ evolves to 
 $U(t)v=\sum_r e^{-it\theta_r}E_rv$ at time $t$. 
Therefore the pair $(\theta_r, E_r)$ with $E_rv=0$ does not influence the state.
We define the \textsl{eigenvalue support} of the vector $v$ to be 
$\Phi_v=\{\theta_r: E_r v \neq 0\}$. 
In the case $v=e_a$ for some vertex $a$, 
we also call $\Phi_{e_a}$ ($\Phi_a$ for short) the eigenvalue support of $a$. 

{\em Perfect state transfer} from vertex $a$ to vertex $b$ occurs at time $\tau$ if 
\begin{equation}
\label{Eqn:defPST}
U(\tau) e_a=\alpha e_b,
\end{equation}
for some phase factor $\alpha$. 
If $a=b$ then we say the quantum walk is {\em periodic at $a$}.

Multiplying $E_r$ to both sides of Equation~(\ref{Eqn:defPST}) gives
\begin{equation}
\label{Eqn:PST}
e^{-\ii \tau \theta_r} E_r e_a = \alpha E_r e_b.
\end{equation}
Hence, for $r=1,\ldots,d$, there exists $q_r(a,b) \in [0, 2\pi)$ such that
\begin{equation}
\label{Eqn:Quarrel}
E_r e_a = e^{\ii q_r(a,b)} E_r e_b.
\end{equation}
We say the vertices $a$ and $b$ are {\em strongly cospectral} when this condition is satisfied, and 
call $q_r(a,b)$  the {\em quarrel from $a$ to $b$ relative to the eigenvalue $\theta_r$}.
Note that strongly cospectral vertices have the same eigenvalue support.

We study perfect state transfer in oriented graphs and in Hermitian graphs in Sections~\ref{Section:PSTOriented} and \ref{Section:PSTHermitian}.
We give here a characterization of perfect state transfer in Hermitian graphs.
\begin{theorem}
\label{Thm:HermitianPST}
Perfect state transfer occurs from $a$ to $b$ in a Hermitian graph $X$
if and only if
\begin{enumerate}[i.]
\item
\label{Cond:HermitionPST1}
$a$ and $b$ are strongly cospectral vertices with quarrels $q_r(a,b)$, for $\theta_r\in \Phi_a$, and
\item
\label{Cond:HermitionPST2}
for $\theta_r, \theta_s, \theta_h, \theta_{\ell}\in \Phi_a$ such that $h\neq \ell$,
there exist integers $m_{r,s}$ and $m_{h,\ell}$ satisfying
\begin{equation*}
\frac{\theta_r-\theta_s}{\theta_h-\theta_{\ell}} = \frac{q_r(a,b)-q_s(a,b)+2m_{r,s}\pi}{q_h(a,b)-q_{\ell}(a,b)+2m_{h,\ell}\pi}.
\end{equation*}
\end{enumerate}
\end{theorem}
\begin{proof}
From Equation~(\ref{Eqn:Quarrel}), we see that perfect state transfer from $a$ to $b$ implies they are strongly cospectral.

Suppose $a$ and $b$ are strongly cospectral with quarrel $q_r(a,b)$, for $\theta_r\in \Phi_a (=\Phi_b)$.
Then Equation~(\ref{Eqn:defPST}) holds if and only if for $\theta_r, \theta_s \in \Phi_a$,
\begin{equation}
\label{Eqn:HermitianPST}
\alpha = e^{\ii \left(q_r(a,b)-\tau \theta_r\right)}=e^{\ii \left(q_s(a,b)-\tau \theta_s\right)}.
\end{equation}
This is equivalent to
\begin{equation*}
e^{\ii \tau \left(\theta_r-\theta_s\right)} = e^{\ii \left(q_r(a,b)-q_s(a,b)\right)}
\end{equation*}
and 
\begin{equation*}
 \tau \left(\theta_r-\theta_s\right) = q_r(a,b)-q_s(a,b) + 2m_{r,s}\pi,
\end{equation*}
for some integer $m_{r,s}$.  Condition~(\ref{Cond:HermitionPST2}) follows immediately.
\end{proof}

We say
the {\em ratio condition on $\Phi_a$} holds if
\begin{equation}
\label{Eqn:Ratio}
\frac{\theta_r-\theta_s}{\theta_h-\theta_{\ell}} \in \QQ
\end{equation}
for $\theta_r, \theta_s, \theta_h, \theta_{\ell}\in \Phi_a$ such that $h\neq \ell$.

\begin{theorem}
\label{Thm:Periodicity}
In a Hermitian graph $X$, $a$ is periodic if and only if the ratio condition on $\Phi_a$ holds.
\end{theorem}
\begin{proof}
Note that $q_r(a,a)=0$ for $\theta_r\in \Phi_a$.  The result follows immediately from Theorem~\ref{Thm:HermitianPST}.
\end{proof}

In Section~\ref{Section:MPGST}, we consider a relaxation of perfect state transfer.  A graph has {\em pretty good state transfer} from $a$ to $b$ if,
for any $\varepsilon > 0$, there is a time $\tau$ satisfying
\begin{equation}
\label{Eqn:defPGST}
\vert U(\tau)_{a,b}\vert \geq 1-\varepsilon.
\end{equation}
Using the proof of Lemma~13.1 in \cite{g12}, we conclude that if there is pretty good state transfer from $a$ to $b$ then $a$ and $b$ are strongly cospectral.
From
\begin{equation*}
U(t)_{a,b} = \sum_{r=1}^d e^{-\ii t \theta_r}e_a^T E_re_b = \sum_{r=1}^d e^{\ii \left(q_r(a,b)-t\theta_r\right)}(E_r)_{b,b},
\end{equation*}
we see that there is pretty good state transfer from $a$ to $b$ if and only if  for any $\epsilon >0$, there exists $\tau>0$ and $\delta_{\epsilon} \in \RR$ such that
\begin{equation*}
\vert \tau \theta_r - q_r(a,b) - \delta_{\epsilon} \vert < \epsilon \pmod{2\pi}, \quad \text{for $r\in \Phi_a$.}
\end{equation*}

\begin{theorem} (Kronecker \cite{lz})
\label{Thm:Kron}
Let $\theta_1,\ldots,\theta_d$ and $q_1,\ldots,q_d$ be arbitrary real numbers.
For any $\epsilon>0$, the system of inequalities
\begin{equation*}
	|\theta_r\tau - q_r| < \epsilon \pmod{2\pi}, 
	\ \ \
	r=1,\ldots,d
\end{equation*}
admits a solution for $\tau$ if and only if, for all set of  integers $l_1,\ldots,l_d$, 
\begin{equation*}
	l_1\theta_1 + \ldots + l_d\theta_d = 0
\end{equation*}
implies
\begin{equation*}
	l_1 q_1 + \ldots + l_d q_d = 0\pmod{2\pi}.
\end{equation*}
\end{theorem}

\begin{theorem}
\label{Thm:PGST}
Let $X$ be Hermitian graph with eigenvalues $\theta_1, \ldots, \theta_d \in \Phi_a$.
Then $X$ has pretty good state transfer from $a$ to $b$ if and only if
the following conditions hold.
\begin{enumerate}[i.]
\item
\label{Cond:PGST1}
The vertices $a$ and $b$ are strongly cospectral with quarrels $q_r(a,b)$, for $r=1,\ldots,d$.
\item
\label{Cond:PGST2}
There exists $\delta\in \RR$ such that, for all integers $l_1,\ldots,l_d$ satisfying
$\sum_{r=1}^d l_r \theta_r = 0$, we have
\begin{equation}
\label{Eqn:PGSTK2}
\sum_{r=1}^d l_r \left(q_r(a,b)+\delta\right) = 0 \pmod{2\pi}.
\end{equation}
\end{enumerate}
\end{theorem}
\begin{proof}
The result follows from Proposition~4.01 of \cite{vanBommel} and Theorem~\ref{Thm:Kron}.
\end{proof}

Let $S$ be a set of vertices in $X$, we say {\em multiple pretty good state transfer} occurs on $S$ if there is pretty good state transfer between any two vertices in $S$.
Section~\ref{Section:MPGST} gives two families of Hermitian graphs that have multiple pretty good state transfer.

\section{Perfect state transfer in oriented graphs}
\label{Section:PSTOriented}

For graphs with real symmetric adjacency matrix, Kay shows that perfect state transfer cannot happen from one vertex to two distinct vertices \cite{k11}.  This monogamous behaviour does not hold in Hermitian graphs with non-real entries.  
A graph has {\em multiple perfect state transfer} on a set $S$ of at least three vertices if there is perfect state transfer between any two vertices in $S$.  When $S=V(X)$, we say $X$ has {\em universal perfect state transfer}.   Lemma~22 of \cite{UPST} gives a construction of Hermitian circulants that admit universal perfect state transfer.  The oriented 3-cycle is a special
case of this construction.  In the same paper, Cameron et al. conjecture that the oriented $K_2$ and the oriented $K_3$ are the only oriented graphs that can have universal perfect state transfer.  We confirm this conjecture in Section~\ref{Subsection:UPST}.

In \cite{gl}, Godsil and Lato investigated multiple perfect state transfer in oriented graph where $S$ is a proper subset of $V(X)$.  They give an example of an oriented graph on eight vertices that admits multiple perfect state transfer on a set of four vertices.
In Section~\ref{Subsection:MPST}, we extend their example to an infinite family of oriented graphs that have multiple perfect state transfer.

\subsection{Universal perfect state transfer}
\label{Subsection:UPST}

In \cite{USTG}, Cameron et al. show that the oriented $K_2$ and $K_3$ with any orientation admit universal perfect state transfer.  They give the following necessary conditions on the Hermitian graphs admitting universal perfect state transfer.
\begin{theorem}
\label{Thm:UPST2014}
Let $H$ be the matrix associated with a Hermitian graph $X$ that admits universal perfect state transfer.
Then the following holds:
\begin{enumerate}
\item
All eigenvalues of $H$ are simple.
\item
If $P$ is a unitary matrix diagonalizing $H$ then $\vert P_{a,b}\vert = \frac{1}{\sqrt{n}}$, for $a,b \in V(X)$.
\item
Every vertex in $X$ is periodic.
\end{enumerate}
\end{theorem}
Suppose $X$ is an oriented graph on $n$ vertices that has universal perfect state transfer.
Let $H$ be its associated Hermitian matrix with spectral decomposition
\begin{equation*}
H = \sum_{r=1}^n \theta_r E_r.
\end{equation*}
Then $E_r$ has rank one with constant diagonal entries $n^{-1}$.  We see that $H^2$ has constant diagonal entries and
the underlying (undirected) graph of $X$ is regular.
Further, it follows from Theorem~6.1 of \cite{gl} that there exists a positive square-free integer $\Delta$ such that $\theta_r \in \ZZ \sqrt{\Delta}$, for $r=1,\ldots,n$. Hence 
\begin{equation}
\label{Eqn:mingap}
 \min_{r\neq s} \vert \theta_r - \theta_s\vert \geq \sqrt{\Delta}.
\end{equation}

We show in the following lemmas that an oriented graph with universal perfect state transfer can have at most eleven vertices.
\begin{lemma}
\label{Lem:difsqsum}
Let $H$ be a Hermitian matrix of order $n$ with zero diagonal entries. Let $\theta_1\leq \theta_2\leq \cdots \leq \theta_n$ be the eigenvalues of $H$.  Then
\begin{equation*}
\sum_{r,s=1}^n \left(\theta_r -\theta_s\right)^2 = 2n \Tr(H^2).
\end{equation*}
\end{lemma}
\begin{proof}
Observe that $\theta_r-\theta_s$ is an eigenvalue of $\left(H \otimes I_n - I_n\otimes H\right)$, for $r,s=1\ldots,n$.
Hence
\begin{equation*}
\sum_{r,s=1}^n \left(\theta_r -\theta_s\right)^2 = \Tr\left(H \otimes I_n - I_n\otimes H\right)^2
= \Tr \left(H^2 \otimes I_n + I_n\otimes H^2 -2H\otimes H \right).
\end{equation*}
The result follows from $\Tr (H\otimes H)=0$.
\end{proof}

\begin{lemma}
\label{Lem:boundsigma}
Let $X$ be an oriented graph on $n$ vertices and $m$ edges
with eigenvalues $\theta_1< \cdots <\theta_n$. 
Let $\sigma = \min_{r\neq s} \vert \theta_r-\theta_s\vert$.  Then
\begin{equation*}
\sigma^2 \frac{n(n^2-1)}{24} \leq m 
\quad \text{and}\quad \sigma^2 \leq \frac{12}{n+1}.
\end{equation*}
\end{lemma}
\begin{proof}
It follows from the definition of $\sigma$ that $\sigma \vert r-s\vert \leq \vert \theta_r-\theta_s\vert$, and
\begin{equation*}
\sigma^2 \sum_{r,s=1}^n (r-s)^2 \leq \sum_{r,s=1}^n \left(\theta_r -\theta_s\right)^2.
\end{equation*}
The lower bound is
\begin{equation*}
\sigma^2 \sum_{r,s=1}^n (r-s)^2 = \sigma^2 \left(2n \sum_{r=1}^n r^2-2 \left(\sum_{r=1}^n r\right)^2 \right) = \sigma^2\frac{n^2(n^2-1)}{6}.
\end{equation*}
Applying Lemma~\ref{Lem:difsqsum} gives
\begin{equation*}
\sigma^2 \frac{n^2(n^2-1)}{6} \leq 2n \Tr(H^2) = 4mn.
\end{equation*}
The second inequality in the lemma follows immediately from $m\leq \binom{n}{2}$.
\end{proof}

\begin{corollary}
\label{Cor:11V}
Let $X$ be an oriented graph on $n$ vertices.  If $X$ admits universal perfect state transfer then $n\leq 11$.
Further, if $n\geq 6$ then $X$ has integral eigenvalues.
\end{corollary}
\begin{proof}
It follows from Equation~(\ref{Eqn:mingap}) that $\sigma^2 \geq \Delta \geq 1$.  The second inequality of 
Lemma~\ref{Lem:boundsigma} gives $n\leq 11$.
When $n\geq 6$, we have $\sigma^2 < 2$ which implies $\Delta=1$ and the eigenvalues of $X$ are integers.
\end{proof}

We are ready to rule out universal perfect state transfer in oriented graphs on more than three vertices.
\begin{theorem}
\label{Thm:orientedUPST}
The oriented $K_2$ and $K_3$ are the only oriented graphs admitting universal perfect state transfer.
\end{theorem}
\begin{proof}
Suppose $X$ is an oriented graph on $n$ vertices that admits universal perfect state transfer.
Then the underlying graph of $X$ is $k$-regular, for some integer $k$.  

Let $\theta_1<\cdots<\theta_n$
be the eigenvalues of the Hermitian matrix $H$ associated with $X$.  Then $\theta_r \in \ZZ \sqrt{\Delta}$, for some 
positive square-free integer $\Delta$.
Since $\ii H$ is a skew-symmetric matrix with entries $\pm 1$, we have
\begin{equation}
\label{Eqn:orientedUPST1}
\theta_r=-\theta_{n+1-r} \quad \text{for $r=1,\ldots,n$.}
\end{equation}
Further, the characteristic polynomial of $\ii H$ is equal to the characteristic polynomial of
its underlying graph over $\ZZ_2$.  

When $n=4$ or $5$, $C_n$ and $K_n$ are the only regular graphs on $n$ vertices.  An exhaustive search rules out oriented graphs on $4$ or $5$ vertices with spectrum satisying the above conditions.

For $n\geq 6$, it follows from Lemma~\ref{Lem:boundsigma} and Corollary~\ref{Cor:11V} that $\sigma=\min_{r\neq s} \vert \theta_r-\theta_s\vert=1$ and
\begin{equation*}
\frac{n^2-1}{12}\leq k \leq n-1.
\end{equation*}
Using this inequality together with the fact that $k$ is even when $n$ is odd, we narrow down to the following possibilities.
\begin{center}
\begin{tabular}{c|c|c|c|c|c|c}
$n$ & 6 & 7 & 8 & 9 & 10 & 11\\
\hline
$k$ & 3, 4, 5 & 4, 6 & 6, 7 & 8 & 9 & 10
\end{tabular}
\end{center}
Applying Equation~(\ref{Eqn:orientedUPST1}) to $\Tr(H^2)$ yields
\begin{equation*}
nk= 2\sum_{r=1}^{\lfloor \frac{n+1}{2}\rfloor} \theta_r^2.
\end{equation*}
Direct computation returns integral solutions to this equation for only three cases:
\begin{center}
\begin{tabular}{c|c|c|l}
$n$ & $k$ & underlying graph & Possible spectrum of $\ii H$ \\
\hline
11 & 10 & $K_{11}$ & $0, \pm \ii, \pm 2\ii, \pm 3\ii, \pm 4\ii, \pm 5\ii$\\
7 & 6 &$K_7$ & $0, \pm \ii, \pm 2\ii,  \pm 4\ii$\\
7 & 4 & $\overline{C_7}$& $0, \pm \ii, \pm 2\ii,  \pm 3\ii$\\
\end{tabular}
\end{center}
It is straightforward to check that for each case, the characteristic polynomial of the underlying graph is not equal to
the polynomial with the roots listed in the table over $\ZZ_2$.  

We conclude that there is no oriented graph on $n\geq 4$ vertices admitting universal perfect state transfer.
\end{proof}
\subsection{Multiple perfect state transfer}
\label{Subsection:MPST}

In \cite{gl}, Godsil and Lato relax the notion of universal perfect state transfer to multiple perfect state transfer on a subset of vertices in oriented graphs.  
Let
\begin{equation*}
H_{\overrightarrow{C}_4}= \begin{bmatrix}0&-\ii&0&\ii\\\ii&0&-\ii&0\\0&\ii&0&-\ii\\-\ii&0&\ii&0\end{bmatrix}
\end{equation*}
be the Hermitian matrix of the directed 4-cycle.  They show that the oriented graph with Hermitian matrix
\begin{equation*}
\begin{bmatrix} 1&0\\ 0&1 \end{bmatrix} \otimes H_{\overrightarrow{C}_4} + \begin{bmatrix} 0&\ii\\ -\ii&0\end{bmatrix} \otimes J_4
\end{equation*}
has multiple perfect state transfer on a set of four vertices.

Making use of this technical lemma from \cite{GSCQW}, we extend the above example to an infinite family of oriented graphs where multiple perfect state transfer occur.
\begin{lemma}
\label{Lem:Tensor}
Let $A$ and $B$ be Hermitian matrices where $A$ has spectral decomposition $A=\sum_r \theta_r E_r$.
Then 
\begin{equation*}
e^{-it(A\otimes B)}=\sum_r E_r\otimes e^{-it\theta_r B}.
\end{equation*}
\end{lemma}

\begin{lemma}
\label{Lem:MPSTconstruction}
Suppose $X$ is an oriented graph on $n$ vertices with associated Hermitian matrix $H_X$, whose eigenvalues are odd integers.
Let $Y$ be the oriented graph with Hermitian matrix
\begin{equation*}
H_Y = I_n \otimes H_{\overrightarrow{C}_4} + H_X \otimes J_4.
\end{equation*}
Then $Y$ admits multiple perfect state transfer on the set $\{4h+1, 4h+2, 4h+3,4h+4\}$, for $h=0,1,\ldots,n-1$.
\end{lemma}
\begin{proof}
Let $H_X = \sum_{r} \theta_r E_r$ be the spectral decomposition of $H_X$.
Since $I_n \otimes H_{\overrightarrow{C}_4} $ and $H_X\otimes J_4$ commute, 
applying Lemma~\ref{Lem:Tensor} gives
\begin{equation*}
e^{-\ii t H_Y} = \left(I_n \otimes e^{-\ii t H_{\overrightarrow{C}_4}}\right) \left(\sum_r E_r \otimes e^{-\ii t \theta_r J_4} \right) = \sum_r E_r \otimes e^{-\ii t \left(H_{\overrightarrow{C}_4}+\theta_rJ_4\right)}.
\end{equation*}
For odd integer $\theta_r$, we have
\begin{equation*}
e^{-\ii \frac{\pi}{4} \left(H_{\overrightarrow{C}_4}+\theta_rJ_4\right)} =
\begin{bmatrix}0&-1&0&0\\0&0&-1&0\\0&0&0&-1\\-1&0&0&0\end{bmatrix}.
\end{equation*}
Hence 
\begin{equation*}
e^{-\ii \frac{\pi}{4} H_Y} = I_n \otimes \begin{bmatrix}0&-1&0&0\\0&0&-1&0\\0&0&0&-1\\-1&0&0&0\end{bmatrix},
\end{equation*}
and, for $h=0, 1, \ldots, n-1$, the vertex $4h+1$ has perfect state transfer to $4h+4$, $4h+3$ and $4h+2$ at time $\frac{\pi}{4}$, $\frac{\pi}{2}$
and $\frac{3\pi}{4}$, respectively.
\end{proof}

If $X$ is obtained by orienting all edges in the $(2m+1)$-cube from one bipartition to the other bipartition, then its associated matrix has the form
\begin{equation*}
H_X=\begin{bmatrix} 0 & \ii B\\ -\ii B^T &0\end{bmatrix}.
\end{equation*}
Then $H_X$ has the same spectrum as the adjacency matrix of the (undirected) $(2m+1)$-cube, which consists of only odd integers.  Lemma~\ref{Lem:MPSTconstruction} gives an oriented graph admitting multiple perfect state transfer for integer $m\geq 0$.  When $m=0$, then $Y$ is the oriented graph given in \cite{gl}.
\section{Perfect state transfer in Hermitian graphs}
\label{Section:PSTHermitian}

We focus on Hermitian graphs with algebraic entries in the first part of this section.
In particular, we study the phase factors when perfect state transfer occurs in these graphs in Section~\ref{Subsection:Phase}.

Suppose $X$ is a Hermitian graph with algebraic entries. By Theorem~6.1 of \cite{RST} and Theorem~\ref{Thm:Periodicity}, if perfect state transfer from $a$ to $b$ occurs then the quantum walk on $X$ is periodic at both $a$ and $b$.
Section~\ref{Subsection:1wayPST} gives examples of Hermitian graphs (with transcendental entries) in which perfect state transfer occurs from $a$ to $b$ but $a$ and $b$ are not periodic.
\subsection{Phase factor}
\label{Subsection:Phase}

We restrict our attention to Hermitian graphs with algebraic entries and extract information about the phase factor when perfect state transfer occurs. 

Let $H$ be an algebraic Hermitian matrix. Its characteristic polynomial has algebraic coefficients.
Given spectral decomposition $H=\sum_r \theta_r E_r$, the eigenvalues $\theta_r$'s are algebraic so are the entries in $E_r$.

\begin{theorem}
\label{Thm:Phase1}
Let $H$ be an algebraic matrix associated with a Hermitian graph with spectral decomposition $H=\sum_r \theta_r E_r$.  If perfect state transfer occurs from $a$ to $b$ with phase factor $\alpha$, then $\alpha$ is algebraic if and only if 
\begin{equation*}
\frac{\theta_r}{\theta_s} \in \QQ,
\end{equation*}
for $\theta_r, \theta_s \in \Phi_a$ such that $\theta_s\neq 0$.
\end{theorem}
\begin{proof}
Suppose perfect state transfer occurs from $a$ to $b$ at time $\tau$ with algebraic phase factor $\alpha$.  It follows from
Equation~(\ref{Eqn:PST}) that $e^{-\ii \tau \theta_r}$ is algebraic, for $\theta_r\in \Phi_a=\Phi_b$.
Applying the Gelfond-Schneider Theorem to
\begin{equation*}
\left(e^{-\ii \tau \theta_s}\right)^{\frac{\theta_r}{\theta_s}} = e^{-\ii \tau \theta_r},
\end{equation*}
for $\theta_r, \theta_s \in \Phi_a$ with $\theta_s\neq 0$, we conclude that $\frac{\theta_r}{\theta_s}$ is rational.

Now suppose $\frac{\theta_s}{\theta_r}\in \QQ$ for $\theta_r, \theta_s \in \Phi_a$ with $\theta_s\neq 0$.
Let $q_r(a,b)$ be the quarrels from $a$ to $b$ relative to $\theta_r\in \Phi_a$. It follows from
Equation~(\ref{Eqn:Quarrel}) that $e^{\ii q_r(a,b)}$ is algebraic.
Applying Equation~(\ref{Eqn:HermitianPST}) yields
\begin{equation*}
\alpha^{\left(\frac{\theta_r}{\theta_s}-1\right)} 
= \left(e^{\ii(q_s(a,b)-\tau \theta_s)}\right)^{\frac{\theta_r}{\theta_s}} e^{\ii(\tau\theta_r - q_r(a,b))}
=\left(e^{\ii q_s(a,b)}\right)^{\frac{\theta_r}{\theta_s}} e^{-\ii q_r(a,b)}.
\end{equation*}
The right-hand side is algebraic, so is  $\alpha$.
\end{proof}

\begin{theorem}
\label{Thm:Phase2}
Let $H$ be an algebraic matrix associated with a Hermitian graph with spectral decomposition $H=\sum_r \theta_r E_r$.
Suppose perfect state transfer occurs from $a$ to $b$ with phase factor $\alpha$.
If there exist integers $k_r$'s satisfying
\begin{equation*}
\sum_{r\in \Phi_a} k_r \theta_r = 0
\quad
\text{and}
\quad
\sum_{r\in \Phi_a} k_r \neq 0
\end{equation*}
then $\alpha$ is algebraic.
\end{theorem}
\begin{proof}
From Equation~(\ref{Eqn:HermitianPST}), we have
\begin{equation*}
\alpha^{\sum_{r\in \Phi_a} k_r } = 
e^{-\ii \tau \left(\sum_{r\in \Phi_a} k_r \theta_r\right)}\prod_{r\in \Phi_a} \left(e^{\ii q_r(a,b)}\right)^{k_r}
= \prod_{r\in \Phi_a} \left(e^{\ii q_r(a,b)}\right)^{k_r}.
\end{equation*}
Since the right-hand side is algebraic and $\sum_{r\in \Phi_a} k_r \neq 0$, we conclude that $\alpha$ is algebraic.
\end{proof}

We apply the theorem to algebraic Hermitian graphs where $\Phi_a$ contains all eigenvalues of $H$.
\begin{corollary}
\label{Cor:Phase}
Let $H$ be an algebraic matrix associated with a Hermitian graph with zero diagonal entries.
Suppose perfect state transfer occurs from $a$ to $b$ with phase factor $\alpha$.
If $a$ has full eigenvalue support then $\alpha$ is algebraic.
\end{corollary}
\begin{proof}
Let $k_r$ be the multiplicity of $\theta_r$, for $\theta_r \in \Phi_a$.
Since $\Phi_a$ contains all eigenvalues of $H$, we have $\sum_{r\in \Phi_a} k_r \theta_r= \Tr(H)=0$ and
$\sum_{r\in \Phi_a}k_r$ equals the number of vertices.  It follows from Theorem~\ref{Thm:Phase2} that the phase factor at perfect state transfer is algebraic.
\end{proof}

Given spectral decomposition of an algebraic Hermitian matrix $H=\sum_r \theta_r E_r$, if $E_r$ has constant diagonal then
every vertex has full eigenvalue support.  In particular, Corollary~\ref{Cor:Phase} applies to
\begin{itemize}
\item
the adjacency matrix of a walk regular graph,
\item
an algebraic Hermitian matrix with zero diagonal that belongs to a Bose-Mesner algebra, and
\item
Hermitian circulants with algebraic entries and zero diagonal.
\end{itemize}
\subsection{One-way perfect state transfer}
\label{Subsection:1wayPST}
We saw at the beginning of Section~\ref{Section:PSTHermitian} that if perfect state transfer occurs from $a$ to $b$ in an algebraic Hermitian graph then both $a$ and $b$ are periodic.  In particular, there is perfect state transfer from $b$ back to $a$.

We give a family of Hermitian graphs, with transcendental entries, that have perfect state transfer from $a$ to $b$ but not periodic at $a$ nor $b$.  In particular, they do not have perfect state transfer from $b$ to $a$.

\begin{theorem}
\label{Thm:1wayPST}
There exist infintely many Hermitian graphs which admit perfect state transfer from $a$ to $b$ but are not periodic at $a$.
\end{theorem}
\begin{proof}
Let $\lambda$ be any real number such that $\lambda \not \in \QQ \pi$.
Define matrices
\begin{equation*}
P=\frac{1}{2}\begin{bmatrix}1&1&1&1\\1&1&-1&-1\\1&-1&e^{\ii\lambda}& -e^{\ii\lambda}\\1&-1&-e^{\ii\lambda}& e^{\ii\lambda}\end{bmatrix} \quad \text{and}\quad
D=\begin{bmatrix}0&0&0&0\\0&\pi&0&0\\0&0&\lambda&0\\0&0&0&\lambda+\pi \end{bmatrix}.
\end{equation*}
Consider the Hermitian matrix
\begin{equation*}
H := PDP^{-1} = \left(\frac{\pi+\lambda}{2}\right) I_4 - 
\begin{bmatrix} 0&\frac{\lambda}{2}&\frac{\pi}{4}(1+e^{-i\lambda}) & \frac{\pi}{4}(1-e^{-i\lambda})\\
\frac{\lambda}{2}&0&\frac{\pi}{4}(1-e^{-i\lambda})&\frac{\pi}{4}(1+e^{-i\lambda})\\
\frac{\pi}{4}(1+e^{i\lambda})&\frac{\pi}{4}(1-e^{i\lambda})&0&\frac{\lambda}{2}\\
\frac{\pi}{4}(1-e^{i\lambda})&\frac{\pi}{4}(1+e^{i\lambda})&\frac{\lambda}{2}&0
\end{bmatrix}.
\end{equation*}

Let $\theta_1=0,\theta_2=\pi,\theta_3=\lambda$ and $\theta_4=\lambda+\pi$.
All vertices have full eigenvalue support.  
Vertices 1 and 3 are strongly cospectral with quarrels: $q_1(3,1)=0$, $q_2(3,1)=\pi$, $q_3(3,1)=\lambda$, and $q_4(3,1)=\lambda+\pi$.
By Theorem~\ref{Thm:HermitianPST}, we have perfect state transfer from vertex $3$ to $1$ at time $\tau=1$ with phase factor $1$.  
As $\lambda$ is not a rational multiple of $\pi$,  we have 
\begin{equation*}
\frac{\theta_3-\theta_1}{\theta_2-\theta_1}=\frac{\lambda}{\pi}\notin\mathbb{Q}.
\end{equation*}
By Theorem~\ref{Thm:Periodicity}, $H$ is not periodic at vertex 1 nor at  vertex 3. 
\end{proof}

\begin{example}
Consider the complex Hadamard matrix
\begin{equation*}
P=\begin{bmatrix}
        1 & 1 & 1 & 1 & i & i & i & i\\
        1 & -1 & e^{i\theta} & -e^{i\theta} & -1 & 1 & -e^{i\theta} & e^{i\theta}\\
        1 & 1 & e^{i2\theta} & e^{i2\theta} & -i & -i & -ie^{i2\theta} & -ie^{i2\theta}\\
        1 & -1 & e^{i3\theta} & -e^{i3\theta} & 1 & -1 & e^{i3\theta} & -e^{i3\theta}\\
        i & i & -i & -i & -1 & -1 & 1 & 1\\
        -i & i & ie^{i\theta} & -ie^{i\theta} & i & -i & -ie^{i\theta} & ie^{i\theta}\\
        i & i & -ie^{i2\theta} & -ie^{i2\theta} & 1 & 1 & -e^{i2\theta} & -e^{i2\theta}\\
        -i & i & ie^{i3\theta} & -ie^{i3\theta} & -i & i & ie^{i3\theta} & -ie^{i3\theta}\\
\end{bmatrix}
\end{equation*}
and diagonal matrix $D=\operatorname{diag}\left(0, \pi, \theta, \theta+\pi, \frac{\pi}{2}, \frac{3\pi}{2},\theta+\frac{\pi}{2},\theta+\frac{3\pi}{2}\right)$. 
Then the Hermitian graph $X$ with  matrix $H=PDP^{-1}$ 
admit perfect state transfer from vertex 1 to 2 at $t=1$, 
from vertex 1 to 3 at $t=2$, from vertex 1 to 4 at $t=3$. 
Each vertex  has full eigenvalue support, and if   $\theta \notin \QQ \pi$, 
then the ratio condition is not satisfied and $X$ is not periodic at any vertex.
\end{example}

\section{Multiple pretty good state transfer}
\label{Section:MPGST}

Theorem~\ref{Thm:1wayPST} shows that it is possible to have one-way perfect state transfer in Hermitian graph.
We now show that pretty good state transfer in Hermitian graphs goes both ways.

\begin{lemma}
\label{Lem:2wayPGST}
If a Hermitian graph admits pretty good state transfer from $a$ to $b$, then it has pretty good state transfer from $b$ to $a$.
\end{lemma}
\begin{proof}
Suppose $U(t)$ is the transition matrix of a Hermitian graph that has pretty good state transfer from $a$ to $b$.  Then, for $\varepsilon > 0$, there exists a time $\tau_1$  such that $U(\tau_1)e_a = \gamma_1 e_b + \rho_1$, for some phase factor $\gamma_1$ and vector $\rho_1$ with $\Vert \rho_1 \Vert < \frac{\varepsilon}{2}$.

As $U(t)$ is almost periodic, there exists $\tau_2 > \tau_1$ such that $U(\tau_2)e_a = \gamma_2 e_a + \rho_2$, for some phase factor $\gamma_2$ and some vector $\rho_2$ with $\Vert \rho_2 \Vert < \frac{\varepsilon}{2}$.
We have 
\begin{equation*}
 U(\tau_2-\tau_1)e_b = \overline{\gamma_1} U(\tau_2) \left( e_a - U(-\tau_1)\rho_1\right) 
 =\overline{\gamma_1}  \left( \gamma_2 e_a +\rho_2 - U(\tau_2-\tau_1)\rho_1\right).
\end{equation*}
Hence
\begin{equation*}
\Vert U(\tau_2-\tau_1)e_b- \overline{\gamma_1}\gamma_2 e_a \Vert 
= \Vert \rho_2 - U(\tau_2-\tau_1)\rho_1 \Vert 
\leq \Vert \rho_1\Vert+\Vert \rho_2\Vert <\varepsilon
\end{equation*}
and there is pretty good state transfer from $b$ to $a$.
\end{proof}

In \cite{zfkwlb}, Zimbor\'{a}s et al. assign a complex weight $e^{\ii \beta}$ to an edge in the following graph and use the weight to control the fidelity at $b$ and $c$ with initial state $e_a$. 
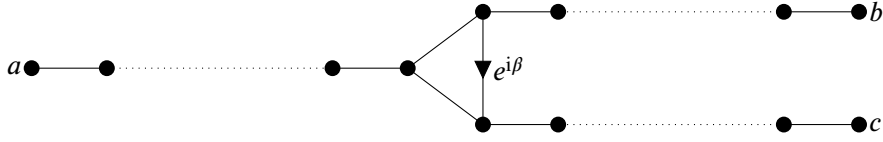
\begin{center}
\begin{tikzpicture}
\path (0,0) coordinate (00);
\fill (00) circle (3pt);
\path (1,0.75) coordinate (01);
\fill (01) circle (3pt);
\path (1,-0.75) coordinate (10);
\fill (10) circle (3pt);

\draw (00) -- (01);
\draw (00) -- (10);
\draw (01)--(10) node[currarrow,pos=0.5, sloped, scale=1.5] {};
\draw  (1,0) node[anchor=west]{{$e^{\ii \beta}$}};

\draw  (-5,0) node[anchor=east]{{$a$}};
\path (-5,0) coordinate (a1);
\fill (a1) circle (3pt);
\path (-4,0) coordinate (a2);
\fill (a2) circle (3pt);
\path (-1,0) coordinate (a3);
\fill (a3) circle (3pt);
\draw (a1)--(a2);
\draw (00)--(a3);
\draw [dotted] (a2)--(a3);

\draw  (6,0.75) node[anchor=west]{{$b$}};
\path (6,0.75) coordinate (b1);
\fill (b1) circle (3pt);
\path (5,0.75) coordinate (b2);
\fill (b2) circle (3pt);
\path (2,0.75) coordinate (b3);
\fill (b3) circle (3pt);
\draw (b1)--(b2);
\draw (01)--(b3);
\draw [dotted] (b2)--(b3);

\draw  (6,-0.75) node[anchor=west]{{$c$}};
\path (6,-0.75) coordinate (c1);
\fill (c1) circle (3pt);
\path (5,-0.75) coordinate (c2);
\fill (c2) circle (3pt);
\path (2,-0.75) coordinate (c3);
\fill (c3) circle (3pt);
\draw (c1)--(c2);
\draw (10)--(c3);
\draw [dotted] (c2)--(c3);
\end{tikzpicture}
\end{center}
This graph can be viewed as the rooted product of the weighted $K_3$ with a path.
Given  a graph $X$ on $n$ vertices and  a rooted graph $Y$ with root $a$. 
The rooted product of $X$ and $Y$, $X\circ Y$, is obtained by taking $n$ isomorphic copies of $Y$ and identifying the $j$-th vertex of $X$ with the root of the $j$-th copy of $Y$.
In this section, we give two families of rooted products that have multiple pretty good state transfer.

\subsection{Oriented 3-cycle rooted with a star}
\label{Subsection:RootedStar}

In \cite{Fan_2012}, Fan and Godsil show that the double star, the rooted product of $K_2$ and $K_{1,m}$, has pretty good state transfer between the two non-pendant vertices if and only if $4m+1$ is not a perfect square.  Note that $K_2$ is the only simple undirected graph with universal perfect state transfer.  
We extend their result to the rooted product of the oriented 3-cycle $\overrightarrow{K}_3$ with $\Star{m}$, where $\Star{m}$ denotes the star $K_{1,m}$ with the non-pendant vertex being its root.

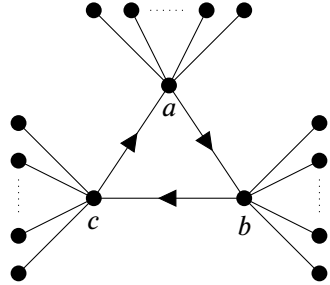
\begin{center}
\begin{tikzpicture}

\path (-1,0) coordinate (a);
\fill (a) circle (3pt);
\draw  (-1,-0.1) node[anchor=north]{{$c$}};
\path (0,1.5) coordinate (b);
\fill (b) circle (3pt);
\draw  (0,1.4) node[anchor=north]{{$a$}};
\path (1,0) coordinate (c);
\fill (c) circle (3pt);
\draw  (1,-0.1) node[anchor=north]{{$b$}};

\draw (a)--(b) node[currarrow,pos=0.5, sloped, scale=1.5] {};
\draw (b)--(c) node[currarrow,pos=0.5, sloped, scale=1.5] {};
\draw (c)--(a) node[currarrow,pos=0.5, xscale=-1, sloped, scale=1.5] {};

\path (-2,1) coordinate (a1);
\fill (a1) circle (3pt);
\path (-2,0.5) coordinate (a2);
\fill (a2) circle (3pt);
\path (-2,-0.5) coordinate (a3);
\fill (a3) circle (3pt);
\path (-2,-1) coordinate (a4);
\fill (a4) circle (3pt);
\draw (a)--(a1);
\draw (a)--(a2);
\draw (a)--(a3);
\draw (a)--(a4);
\draw [dotted] (-2,0.25)--(-2,-0.25);

\path (-1,2.5) coordinate (b1);
\fill (b1) circle (3pt);
\path (-0.5,2.5) coordinate (b2);
\fill (b2) circle (3pt);
\path (0.5,2.5) coordinate (b3);
\fill (b3) circle (3pt);
\path (1,2.5) coordinate (b4);
\fill (b4) circle (3pt);
\draw (b)--(b1);
\draw (b)--(b2);
\draw (b)--(b3);
\draw (b)--(b4);
\draw [dotted] (-0.25,2.5)--(0.25,2.5);

\path (2,1) coordinate (c1);
\fill (c1) circle (3pt);
\path (2,0.5) coordinate (c2);
\fill (c2) circle (3pt);
\path (2,-0.5) coordinate (c3);
\fill (c3) circle (3pt);
\path (2,-1) coordinate (c4);
\fill (c4) circle (3pt);
\draw (c)--(c1);
\draw (c)--(c2);
\draw (c)--(c3);
\draw (c)--(c4);
\draw [dotted] (2,0.25)--(2,-0.25);
\end{tikzpicture}
\end{center}

\begin{lemma}
\label{Lem:RootStar}
Suppose $a$ and $b$ are strongly cospectral vertices in the Hermitian graph $X$ on $n\geq 2$ vertices.
Then they are strongly cospectral in the rooted product $X\circ \Star{m}$.
\end{lemma}
\begin{proof}
Let $H_X$ be the Hermitian matrix associated with $X$ with spectral decomposition $H_X=\sum_{r=1}^d \theta_r E_r$ .
Then the matrix associated with the rooted product $Y=X\circ \Star{m}$ is 
\begin{equation*}
H_Y = \begin{bmatrix} 1&0&0&\cdots&0 \\  0&0&0&\cdots&0 \\  \vdots& \vdots & \vdots &\ddots & \vdots\\ 0&0&0&\cdots&0 \\0&0&0&\cdots&0 \end{bmatrix}  \otimes H_X + \begin{bmatrix} 0&1&1&\cdots&1 \\  1 & 0&0&\cdots&0\\  \vdots& \vdots & \vdots &\ddots & \vdots\\1 & 0&0&\cdots&0\\ 1 & 0&0&\cdots&0\end{bmatrix} \otimes I_n.
\end{equation*}
For $r=1,\ldots, d$, define
\begin{equation*}
\lambda_r^{\pm} = \frac{\theta_r \pm \sqrt{\theta_r^2+4m}}{2},
\end{equation*}
and
\begin{equation*}
F_r^{\pm} = \frac{1}{(\lambda_r^{\pm})^2+m}
\begin{bmatrix} (\lambda_r^{\pm})^2 &\lambda_r^{\pm}&\lambda_r^{\pm}&\cdots&\lambda_r^{\pm} \\  \lambda_r^{\pm}&1&1&\cdots&1 \\  \vdots& \vdots & \vdots &\ddots & \vdots\\ \lambda_r^{\pm}&1&1&\cdots&1 \\\lambda_r^{\pm}&1&1&\cdots&1 \\\end{bmatrix}
 \otimes E_r.
\end{equation*}
Define
\begin{equation*}
F_0 = \begin{bmatrix} 0 & &\0_m\\ &&\\ \0_m &&I_m-\frac{1}{m}J_m\end{bmatrix}\otimes I_n.
\end{equation*}
Then $H_Y$ has spectral decomposition
\begin{equation}
\label{Eqn:RootStar}
H_Y= 0 \cdot F_0 + \sum_{r=1}^d \left(\lambda_r^+ \cdot F_r^+ + \lambda_r^- \cdot F_r^-\right).
\end{equation}
Note that the $(1,1)$-block are indexed by the vertices in $X$ and the eigenvalue $0$ is not in the support of $a$ nor $b$.
The result follows from the $(1,1)$-block of $F_r^+$ and $F_r^-$ being non-zero scalar multiple of $E_r$.
\end{proof}

\begin{corollary}
Suppose $X$ is a Hermitian graph with universal perfect state transfer with spectrum $\Phi$.  
Let $S$ be the set of non-pendant vertices in $X\circ \Star{m}$.
Let 
\begin{equation*}
\Psi = \left\{\frac{\theta \pm \sqrt{\theta^2+4m}}{2} \ \big\vert\ \theta \in \Phi\right\}.
\end{equation*}
If $\Psi$ is linearly independent over $\QQ$, then $X\circ \Star{m}$ has multiple pretty good state transfer on $S$.
\end{corollary}
\begin{proof}
For $a,b \in S$, there is perfect state transfer between $a$ and $b$ in $X$, 
so $a$ and $b$ are strongly cospectral in $X\circ \Star{m}$ by Lemma~\ref{Lem:RootStar}. 
We see in Equation~(\ref{Eqn:RootStar}) that $\Psi$ is the eigenvalue support of $a$ in the rooted product.  It follows from Theorem~\ref{Thm:PGST} that pretty good state transfer occurs between
$a$ and $b$ in $X\circ \Star{m}$.
\end{proof}

In the following result, we focus on $X=\overrightarrow{K}_3$ which has spectral decomposition
\begin{equation*}
 \begin{bmatrix} 0&-\ii&\ii\\ \ii &0&-\ii \\ -\ii&\ii&0\end{bmatrix}
=0 \cdot \frac{1}{3} J_3  + \sqrt{3} \cdot \frac{1}{3} \begin{bmatrix}1&e^{-2\pi\ii/3}&e^{2\pi\ii/3}\\e^{2\pi\ii/3}&1&e^{-2\pi\ii/3}\\ e^{-2\pi\ii/3}& e^{2\pi\ii/3} & 1 \end{bmatrix}
- \sqrt{3} \cdot \frac{1}{3} \begin{bmatrix}1&e^{2\pi\ii/3}&e^{-2\pi\ii/3}\\e^{-2\pi\ii/3}&1&e^{2\pi\ii/3}\\ e^{2\pi\ii/3}& e^{-2\pi\ii/3} & 1\end{bmatrix}.
\end{equation*}
Hence any two vertices in $\overrightarrow{K}_3$ are strongly cospectral.  Let $V(\overrightarrow{K}_3)=\{a,b,c\}$.
Then the eigenvalue support of $a$ in $\overrightarrow{K}_3 \circ \Star{m}$ are $\lambda_1= \sqrt{m}$, $\lambda_2=-\sqrt{m}$,
\begin{equation*}
 \lambda_3=\frac{\sqrt{3}+\sqrt{3+4m}}{2},\quad
\lambda_4=\frac{\sqrt{3}-\sqrt{3+4m}}{2}, \quad
\lambda_5=\frac{-\sqrt{3}+\sqrt{3+4m}}{2} \quad\text{and}\quad \lambda_6=\frac{-\sqrt{3}-\sqrt{3+4m}}{2}.
\end{equation*}
From Equation~(\ref{Eqn:RootStar}), the quarrels in $\overrightarrow{K}_3 \circ \Star{m}$ are
\begin{equation*}
q_r(a,b) = 
\begin{cases}
0 & \text{if $r=1,2$,}\\
\frac{2\pi}{3} & \text{if $r=3,4$, and}\\
-\frac{2\pi}{3} & \text{if $r=5,6$.}
\end{cases}
\end{equation*}
\ignore{
and
\begin{equation*}
q_r(a,b) = 
\begin{cases}
0 & \text{if $r=1,2$,}\\
-\frac{2\pi}{3} & \text{if $r=3,4$, and}\\
\frac{2\pi}{3} & \text{if $r=5,6$.}
\end{cases}
\end{equation*}
}

\begin{theorem}
\label{Thm:RootStar}
The rooted product $\overrightarrow{K}_3 \circ \Star{m}$ admits multiple pretty good state transfer on the set
$\{a,b,c\}$ of non-pendant vertices if and only if one of the following holds.
\begin{enumerate}
\item
\label{Cond:RootStar1}
$\gcd(3, m)=1$.
\item
\label{Cond:RootStar2}
$m=3s$, for some integer $s$ such that neither $s$ nor $4s+1$ are perfect square.
\item
\label{Cond:RootStar3}
$m=27 k^2$, for some integer $k$. 
\item
\label{Cond:RootStar4}
$m=27k^2+27k+6$, for some integer $k$.
\end{enumerate}
\end{theorem}
\begin{proof}
Since $\overrightarrow{K}_3 \circ \Star{m}$ has an automorphism that maps $a$ to $b$, $b$ to $c$ and $c$ to $a$, it is sufficient
to prove that there is pretty good state transfer from $a$ to $b$ in the rooted product.

By Lemma~\ref{Lem:RootStar}, Condition~(\ref{Cond:PGST1}) of Theorem~\ref{Thm:PGST} holds.
For Condition~(\ref{Cond:PGST2}) of Theorem~\ref{Thm:PGST}, we consider
integers $l_1, \ldots, l_6$  satisfying
\begin{equation}
\label{Eqn:K1}
\sum_{r=1}^6 l_r \lambda_r=
\left(l_1-l_2\right)\sqrt{m} + \left(\frac{l_3+l_4-l_5-l_6}{2}\right)\sqrt{3} + \left(\frac{l_3-l_4+l_5-l_6}{2}\right)\sqrt{3+4m} = 0.
\end{equation}

\begin{enumerate}[C{a}se 1:]
\item[C{a}se 1:]
If $\gcd(3,m)=1$ then the set $\{\sqrt{3}, \sqrt{m}, \sqrt{3+4m}\}$ is linearly independent over $\QQ$.
Equation~(\ref{Eqn:K1}) implies $(l_3+l_4-l_5-l_6)/2=0$ and
\begin{equation}
\label{Eqn:K2}
\sum_{r=1}^6 l_r q_r(a,b) = \left(l_3+l_4-l_5-l_6\right) \frac{2\pi}{3} = 0 \pmod{2\pi}.
\end{equation}
Condition~(\ref{Cond:PGST2}) of Theorem~\ref{Thm:PGST} holds with $\delta=0$, so there is pretty good state transfer from $a$ to $b$ in $\overrightarrow{K}_3 \circ \Star{m}$.

\item[C{a}se 2:]
When $m=3s$, Equation~(\ref{Eqn:K1}) becomes
\begin{equation*}
\left(l_1-l_2\right)\sqrt{s} + \left(\frac{l_3+l_4-l_5-l_6}{2}\right) + \left(\frac{l_3-l_4+l_5-l_6}{2}\right)\sqrt{1+4s} = 0.
\end{equation*}
If $s$ and $4s+1$ are not perfect squares then $\{1,\sqrt{s}, \sqrt{1+4s}\}$ is linearly independent over $\QQ$ and
Equation~(\ref{Eqn:K1}) implies Equation~(\ref{Eqn:K2}).  Hence there is pretty good state transfer from $a$ to $b$.

\item[C{a}se 3:]
Suppose $m=3h^2$, for some integer $h$.  Then $4h^2+1$ is not a perfect square, and Equation~(\ref{Eqn:K1}) becomes
\begin{equation*}
 \left(\frac{2h(l_1-l_2)+ l_3+l_4-l_5-l_6}{2}\right) + \left(\frac{l_3-l_4+l_5-l_6}{2}\right)\sqrt{4h^2+1} = 0,
\end{equation*}
which implies $l_3+l_4-l_5-l_6 = -2h(l_1-l_2)$.
If $h=3k$, for some integer $k$, then Equation~(\ref{Eqn:K2}) holds and pretty good state transfer occurs from $a$ to $b$.

Suppose $h$ is not divisible by $3$.
Equation~(\ref{Eqn:K1}) holds when $l_1=l_2=l_4=l_5=0$ and $l_3=l_6=1$.
Since
\begin{equation*}
\sum_{r=1}^6 l_r \left(q_r(a,b) + \delta\right)=2\delta,
\end{equation*}
Equation~(\ref{Eqn:PGSTK2}) holds if and only if $\delta \in \ZZ \pi$.

Equation~(\ref{Eqn:K1}) also holds when $l_1=1, l_2=l_3=l_4=0, l_5=l_6=h$, but
\begin{equation*}
\sum_{r=1}^6 l_r (q_r(a,b)+\delta) =  -\frac{4h\pi}{3} + (2h+1)\delta \neq 0 \pmod{2\pi}
\end{equation*}
when $\delta \in \ZZ \pi$.
We conclude that pretty good state transfer from $a$ to $b$ does not occur.

\item[C{a}se 4:]
Suppose $m=3s$ with $4s+1=h^2$, for some integer $h$. Then $s$ is not a perfect square, and Equation~(\ref{Eqn:K1}) becomes
\begin{equation*}
 (l_1-l_2)\sqrt{s} + \frac{(l_3+l_4-l_5-l_6) + h(l_3-l_4+l_5-l_6)}{2}= 0,
\end{equation*}
which implies $l_3+l_4-l_5-l_6 = -h(l_3-l_4+l_5-l_6)$.
If $h$ is divisible by $3$ then Equation~(\ref{Eqn:K2}) holds and pretty good state transfer occurs from $a$ to $b$.
In this case, $m=27k^2+27k+6$ if we write $4s+1 = 3^2(2k+1)^2$.

If $h$ is not divisible by $3$,  Equation~(\ref{Eqn:K1}) holds when $l_1=l_2=l_4=l_5=0$, $l_3=l_6=1$ and when
$l_1=l_2=0$, $l_3=l_4=h$, $l_5=-1$ and $l_6=1$.  Using the same argument as in the previous case, we see that
there does not exist $\delta$ satisfying Equation~(\ref{Eqn:PGSTK2}) for both assignments for the $l_j$'s.
We conclude that pretty good state transfer from $a$ to $b$ does not occur.
\end{enumerate}
\end{proof}

\subsection{Circulants rooted with a looped path}
\label{Subsection:RootedPath}

In \cite{kly17}, Kempton et al. show that a path with a loop on each end-vertex with transcendental weight $\gamma$ has pretty good state transfer between the two end-vertices.  
We use $\lp{m}$ to denote the rooted path on vertices $\{1, 2, \ldots, m\}$ that has root $m$ and a loop on vertex $1$ with weight $\gamma$.
Then the path of length $2m-1$ with a loop of weight $\gamma$ on each end-vertex studied in \cite{kly17} can be viewed as the rooted product of $K_2$ with $\lp{m}$. 

\begin{center}
\begin{tikzpicture}

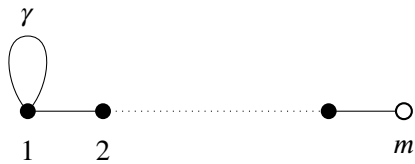
\draw  (2.5,0) node[below]{{Path $\lp{m}$ rooted at $m$ with a loop at $1$}};

\path (0,1) coordinate (b1);
\path (1,1) coordinate (b2);
\path (4,1) coordinate (b3);
\path (5,1) coordinate (b4);

\draw (b1)--(b2);
\draw (b3)--(b4);
\draw [dotted] (b2)--(b3);

\fill (b1) circle (3pt);
\fill (b2) circle (3pt);
\fill (b3) circle (3pt);
\filldraw [thick,fill=white] (b4) circle (3pt);

\draw  (0,0.75) node[anchor=north]{{$1$}};
\draw  (1,0.75) node[anchor=north]{{$2$}};
\draw  (5,0.75) node[anchor=north]{{$m$}};

\draw (b1) to [out=135,in=180] (0,2) to [out=0, in=45] (b1);
\draw  (0,2) node[anchor=south]{{$\gamma$}};

\end{tikzpicture}
\end{center}

We extend their result to the rooted product $X \circ \lp{m}$ where $X$
is Hermitian circulant with rational eigenvalues that admits universal perfect state transfer.
Orthogonal polynomials and field trace are the main tools used in this section.   Please see Chapter 8 of \cite{AC} for the background of orthogonal polynomials, and see \cite{kly17} and Chapter 14 of \cite{df} for some basic facts on field trace.

Suppose $V(X)=\{x_0,x_1,\ldots,x_{n-1}\}$.  Then we label the vertices of $X\circ \lp{m}$ 
with the ordered pair $(x_h, j)$ denoting the $j$-th vertex on $\lp{m}$ that is rooted at $x_h$ in $X$, for $h=0,1,\ldots, n-1$ and $j=1,\ldots, m$.

\begin{center}
\begin{tikzpicture}

\path (0,0) coordinate (00);
\fill (00) circle (3pt);
\draw (-0.25,0) node[above]{$(x_0,m)$};
\path (1,1) coordinate (01);
\fill (01) circle (3pt);
\draw (01) node[above]{$(x_1,m)$};
\path (1,-1) coordinate (10);
\fill (10) circle (3pt);
\draw (10) node[below]{$(x_2,m)$};

\draw (01)--(10) node[currarrow,pos=0.5, sloped, scale=1.5] {};
\draw (10)--(00) node[currarrow,pos=0.5, sloped, scale=1.5] {};
\draw (00)--(01) node[currarrow,pos=0.5, sloped, scale=1.5] {};

\path (-5,0) coordinate (a1);
\fill (a1) circle (3pt);
\draw (a1) node[left]{$(x_0,1)$};
\path (-4,0) coordinate (a2);
\fill (a2) circle (3pt);
\draw (a2) node[below]{$(x_0,2)$};
\path (-1,0) coordinate (a3);
\fill (a3) circle (3pt);
\draw (a1)--(a2);
\draw (00)--(a3);
\draw [dotted] (a2)--(a3);
\draw (a1) to [out=135,in=180] (-5,1) to [out=0, in=45] (a1);
\draw  (-5,1) node[anchor=south]{{$\gamma$}};

\path (6,1) coordinate (b1);
\fill (b1) circle (3pt);
\draw (b1) node[right]{$(x_1,1)$};
\path (5,1) coordinate (b2);
\fill (b2) circle (3pt);
\draw (b2) node[below]{$(x_1,2)$};
\path (2,1) coordinate (b3);
\fill (b3) circle (3pt);
\draw (b1)--(b2);
\draw (01)--(b3);
\draw [dotted] (b2)--(b3);
\draw (b1) to [out=135,in=180] (6,2) to [out=0, in=45] (b1);
\draw  (6,2) node[anchor=south]{{$\gamma$}};

\path (6,-1) coordinate (c1);
\fill (c1) circle (3pt);
\draw (c1) node[right]{$(x_2,1)$};
\path (5,-1) coordinate (c2);
\fill (c2) circle (3pt);
\draw (c2) node[below]{$(x_2,2)$};
\path (2,-1) coordinate (c3);
\fill (c3) circle (3pt);
\draw (c1)--(c2);
\draw (10)--(c3);
\draw [dotted] (c2)--(c3);
\draw (c1) to [out=135,in=180] (6,0) to [out=0, in=45] (c1);
\draw  (6,0) node[anchor=south]{{$\gamma$}};

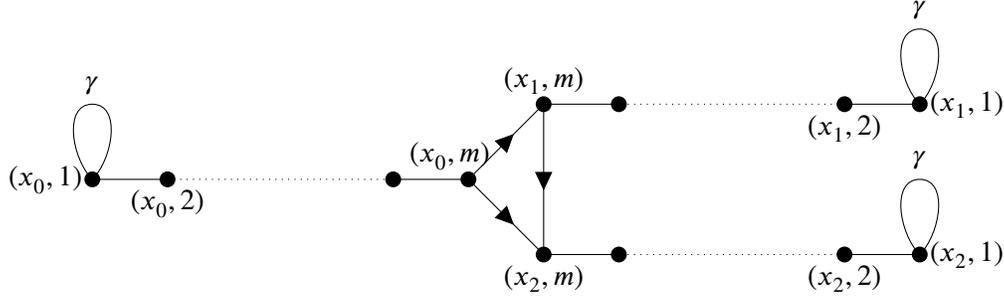
\draw (0,-2) node[below]{The rooted product of $\overrightarrow{K}_3$ with $\lp{m}$};
\end{tikzpicture}
\end{center}

Let $H_X$ be the matrix of the Hermitian circulant $X$ with universal perfect state transfer.
It follows from Theorem~8 of \cite{USTG} that the  eigenvalues of $H_X$ are simple.
Given distinct eigenvalues  $\theta_0, \theta_1, \ldots, \theta_{n-1}$ of $H_X$
and the discrete Fourier matrix of order $n$
\begin{equation*}
F_n = \frac{1}{\sqrt{n}}
\begin{bmatrix}
1&1&1&\cdots & 1\\
1&\zeta&\zeta^2&\cdots & \zeta^{n-1}\\
1&\zeta^2&\zeta^4&\cdots & \zeta^{2(n-1)}\\
\vdots&\vdots&\vdots&\ddots&\vdots\\
1&\zeta^{n-1}&\zeta^{2(n-1)}&\cdots & \zeta^{(n-1)^2}\\
\end{bmatrix}
\end{equation*}
where $\zeta = e^{2\pi \ii/n}$,
we can write
\begin{equation*}
H_X=F_n \begin{bmatrix}\theta_0&0&\cdots&0\\0&\theta_1&\cdots&0\\\vdots&\vdots&\ddots&\vdots\\0&0&\cdots&\theta_{n-1}\end{bmatrix} F_n^*.
\end{equation*}
For $0\leq a,b \leq n-1$, the vertices $x_a$ and $x_b$ are strongly cospectral with quarrel
\begin{equation}
\label{Eqn:XQ}
q_j(x_a, x_b) = \frac{2\pi j (b-a)}{n},
\end{equation}
for $j=0,1,\ldots,n-1$.

Theorem~22 of \cite{USTG} gives the following characterization of Hermitian circulants that have universal perfect state transfer.
\begin{theorem}
\label{Thm:USTCir}
Let $X$ be a Hermitian circulant on $n$ vertices with simple eigenvalues $\theta_0, \ldots, \theta_{n-1}$.
Then  $X$ has universal perfect state transfer if and only if there exist 
$\alpha, \beta \in \RR$ with $\beta>0$, $c_0,\ldots,c_{n-1} \in \ZZ$ and integer $h$ coprime with $n$ such that
\begin{equation*}
\theta_j = \alpha + \beta \left(jh+c_j n\right),
\end{equation*}
for $j=0,\ldots, n-1$.
\end{theorem}

To determine the spectrum of $Z=X\circ \lp{m}$, we consider the $m\times m$ Jacobi matrices
\begin{equation}
\label{Eqn:Tj}
T_j:=\begin{bmatrix}
\gamma & 1 & 0 & \cdots& 0 & 0\\
1&0&1&\cdots & 0& 0\\
0&1&0&\cdots & 0&0\\
\vdots&\vdots&\vdots&\ddots&\vdots&\vdots\\
0&0&0&\cdots &0& 1\\
0&0&0&\cdots &1&  \theta_j
\end{bmatrix},
\qquad \text{for $j=0, 1, \ldots,n-1.$}
\end{equation}
Let $\varphi_{j,0}=1$ and
let $\varphi_{j,r}(t)$ be the characteristic polynomial of the $r$-th leading principal submatrix of $T_j$, for $r=1,\ldots,m$.
Then  $\varphi_{j, 0}(t),\varphi_{j, 1}(t), \ldots, \varphi_{j, m}(t)$ is a sequence of orthogonal polynomials satisfying 
$\varphi_{j, 0}(t)=1$, 
$\varphi_{j, 1}(t)=t-\gamma$,
\begin{equation}
\label{Eqn:OP1}
\varphi_{j, r}(t)= t\ \varphi_{j, r-1}(t) - \varphi_{j, r-2}(t) 
\end{equation}
for $r=2,\ldots,m-1$, and
\begin{equation}
\label{Eqn:OP2}
\varphi_{j, m}(t)= \left(t-\theta_j\right) \varphi_{j, m-1}(t) - \varphi_{j, m-2}(t).
\end{equation}
From Lemma~8.5.2 of \cite{AC}, the roots $\lambda_{j,1}, \ldots, \lambda_{j,m}$ of $\varphi_{j,m}(t)=0$ are the eigenvalues of $T_j$.  Further, 
\begin{equation*}
\Phi_{j,s} = 
\begin{bmatrix}
1&\varphi_{j, 1}(\lambda_{j,s})& \ldots& \varphi_{j, m-1}(\lambda_{j,s})
\end{bmatrix}^T
\end{equation*}
is an eigenvector of $T_j$ corresponding to eigenvalue $\lambda_{j,s}$, for $s=1,\ldots,m$.
It follows from Lemma~8.1.1 of \cite{AC} that the eigenvalues of $T_j$ are simple.  It is also known that consecutive orthogonal polynomials do not have non-trivial common factor.

The Hermitian matrix of $Z$ is
\begin{equation}
\label{Eqn:HZ}
H_Z = 
\begin{bmatrix}
0 & 0 &  \cdots& 0 & 0\\
0&0&\cdots & 0& 0\\
\vdots&\vdots&\ddots&\vdots&\vdots\\
0&0&\cdots &0& 0\\
0&0&\cdots &0& 1\\
\end{bmatrix}
\otimes H_X
+
\begin{bmatrix}
\gamma & 1 &  \cdots& 0 & 0\\
1&0&\cdots & 0& 0\\
\vdots&\vdots&\ddots&\vdots&\vdots\\
0&0&\cdots &0& 1\\
0&0&\cdots &1& 0\\
\end{bmatrix}
\otimes I_n.
\end{equation}
Since $H_X F_n e_j = \theta_j F_ne_j$, we have
\begin{equation}
\label{Eqn:HZev}
H_Z \left(\Phi_{j,s} \otimes F_n e_j \right) = \lambda_{j,s} \left(\Phi_{j,s} \otimes F_n e_j \right)
\end{equation}
for $j=0,\ldots, n-1$ and $s=1,\ldots,m$. 

\begin{lemma}
\label{Lem:X_PmSpecDecomp}
Let $X$ be a Hermitian circulant with distinct eigenvalues $\theta_0, \theta_1, \ldots, \theta_n$ and let $F_n$, $\lambda_{j,s}$, and $\Phi_{j,s}$ be defined as above.
For $j=0,\ldots, n-1$ and $s=1,\ldots,m$,  $\lambda_{j,s}$ is a simple eigenvalue of the Hermitian graph $Z$ defined in Equation~(\ref{Eqn:HZ}), with spectral decomposition
\begin{equation*}
H_Z = \sum_{j=0}^{n-1} \sum_{s=1}^m \lambda_{j,s}  \frac{1}{\Vert \Phi_{j,s}\Vert^2} 
\left(\Phi_{j,s} \Phi_{j,s}^*\right) \otimes \left( (F_ne_j)(F_ne_j)^*\right).
\end{equation*}

For $x_a, x_b \in V(X)$ and $h=1,\ldots,m$, the vertices $(x_a, h)$ and $(x_b,h)$ are strongly cospectral in $Z$ with quarrel corresponding to eigenvalues $\lambda_{j,s}$ being
\begin{equation*}
q_{j,s}\left((x_a,h),(x_b,h)\right) = \frac{2\pi j (b-a)}{n},
\end{equation*}
for $j=0,\ldots, n-1$ and $s=1,\ldots,m$. 
\end{lemma}
\begin{proof}
It is sufficient to show that the eigenvalues $\lambda_{j,s}$ of $Z$, for $j=0,\ldots, n-1$ and $s=1,\ldots,m$, are distinct.

Supoose $\lambda_{j_1,s_1} = \lambda_{j_2,s_2}$.  From Equation~(\ref{Eqn:OP1}), we have
\begin{equation*}
\varphi_{j_1,r}\left(\lambda_{j_1,s_1}\right) = \varphi_{j_2,r} \left(\lambda_{j_2,s_2}\right) ,
\end{equation*}
for $r=1,\ldots,m-1$.
From Equation~(\ref{Eqn:OP2}), $\varphi_{j_1,m}\left(\lambda_{j_1,s_1}\right) = \varphi_{j_2,m}\left(\lambda_{j_2,s_2}\right)=0$ implies $\theta_{j_1}=\theta_{j_2}$ and $j_1=j_2$.  Since $\varphi_{j_1,m}(t)=0$ has $m$ distinct roots, we conclude that $s_1=s_2$.

We get the quarrels of $Z$ directly from Equations~(\ref{Eqn:HZev})~and~(\ref{Eqn:XQ}).
\end{proof}

For the rest of this section, we assume that $\gamma$ is transcendental and $\theta_0,\theta_1,\ldots, \theta_{n-1}\in \QQ$ as in
Theorem~\ref{Thm:MPGSTrooted}.   Applying Laplace expansion along the first two rows of $T_j$ in Equation~(\ref{Eqn:Tj}) gives
\begin{equation*}
\varphi_{j,m}(t) = (t-\gamma)g_{n-1}(t) - g_{n-2}(t),
\end{equation*}
where $g_{n-1}(t)$ is the characteristic polynomial of the $(n-1)\times (n-1)$ Jacobi matrix
\begin{equation*}
\begin{pmatrix} 
\theta_j & 1 & \cdots& 0 & 0\\
1&0&\cdots & 0& 0\\
\vdots&\vdots&\ddots&\vdots&\vdots\\
0&0&\cdots &0& 1\\
0&0&\cdots &1& 0
\end{pmatrix},
\end{equation*}
and $g_{n-2}(t)$ is the characteristic polynomial of its $(n-2)$-th leading principal submatrix.
Now $g_{n-1}(t)$ and $g_{n-2}(t)$ are consecutive orthogonal polynomials, so they do not have any common factor of positive degree.  
Since $g_{n-1}(t)$ and $g_{n-2}(t)$ are rational polynomials and $\gamma$ is transcendental, we conclude that $\varphi_{j,m}(t)$ is irreducible over $\QQ(\gamma)$.  Then the splitting field $F_j$ of $\varphi_{j,m}(t)$ is a Galois extension over $\QQ(\gamma)$.

Given a Galois extension $E/K$, we use $\Tr_{E/K}(\mu)$ to denote the trace of $\mu$ from $E$ to $K$.  
Here are some properties of the trace map useful for the proof of Theorem~\ref{Thm:MPGSTrooted}.
\begin{theorem}
\label{Thm:Trace}
Let $E/K$ be a Galois extension. The following properties hold.
\begin{enumerate}[i.]
\item
\label{Prop:Trace1}
For $\mu \in E$, $\Tr_{E/K}(\mu) \in K$.
\item
\label{Prop:Trace2}
For $\mu \in K$, $\Tr_{E/K}(\mu)= [E\ :\ K]\mu$.
\item
\label{Prop:Trace3}
For $\mu_1, \mu_2 \in E$, $\Tr_{E/K}(\mu_1+\mu_2)=\Tr_{E/K}(\mu_1)+\Tr_{E/K}(\mu_2)$.
\item
\label{Prop:Trace4}
If $K\subset F\subset E$ are extension fields, then $\Tr_{E/K}(\mu)= \Tr_{F/K}\left(\Tr_{E/F}(\mu)\right)$.
\item
\label{Prop:Trace5}
If the minimal polynomial of $\mu \in E$ over $K$ is $t^m+a_{m-1}t^{m-1}+\cdots+c_0$ then
\begin{equation*}
\Tr_{E/K}(\mu) = -\frac{[E\ :\ K]}{m} a_{m-1}.
\end{equation*}
\end{enumerate}
\end{theorem}

The eigenvalue $\lambda_{j,s}$ of $X\circ \lp{m}$ has minimal polynomial $\varphi_{j,m}(t)$ over $\QQ(\gamma)$.
Applying Property~(\ref{Prop:Trace5}) to $\lambda_{j,s}\in F_j$, Equation~(\ref{Eqn:OP2}) gives
\begin{equation}
\label{Eqn:TraceFj}
\Tr_{F_j/\QQ(\gamma)}(\lambda_{j,s}) = \frac{[F_j\ :\ \QQ(\gamma)]}{m} \left(\gamma + \theta_j\right).
\end{equation}
Consider the smallest extension field $M$ of $F_j$ that contains $F_0, \ldots, F_{n-1}$.  For $j=0,\ldots,n-1$, $M/F_j$
is a Galois extension.   It follows from Properties~(\ref{Prop:Trace2}) and (\ref{Prop:Trace4}) and Equation~(\ref{Eqn:TraceFj})
that
\begin{equation}
\label{Eqn:TraceM}
\Tr_{M/\QQ(\gamma)}(\lambda_{j,s}) = \Tr_{F_j/\QQ(\gamma)}\left([M\ :\ F_j]\lambda_{j,s}\right)=
[M\ :\ F_j] \frac{ [F_j\ :\ \QQ(\gamma)]}{m} \left(\gamma + \theta_j\right)
= \frac{[M\ :\ \QQ(\gamma)]}{m} \left(\gamma + \theta_j\right).
\end{equation}

\begin{theorem}
\label{Thm:MPGSTrooted}
Let $X$ be a Hermitian circulant on $n$ vertices that admits universal perfect state transfer with eigenvalues given in Theorem~\ref{Thm:USTCir}.
If $\theta_0, \ldots, \theta_{n-1}\in \QQ$ and $\gamma$ is transcendental then, for any positive integer $m$, the rooted product $X \circ \lp{m}$ has multiple pretty good state transfer on the set $\{(x_0,h), (x_1,h), \ldots, (x_{n-1},h)\}$, for $1\leq h \leq m$.
\end{theorem}
\begin{proof}
For $h=1,\ldots,m$,
$X\circ \lp{m}$ has an automorhism that maps $(x_a,h)$ to $(x_{a+1},h)$, for $a\in \ZZ_n$.  It is sufficient
to show that there is pretty good state transfer from $(x_0,h)$ to $(x_1,h)$.
By Lemma~\ref{Lem:X_PmSpecDecomp}, $(x_0,h)$ and $(x_1,h)$ are strongly cospectral with quarrels
\begin{equation*}
q_{j,s}\left((x_0,h),(x_1,h)\right) = \frac{2\pi j}{n},
\end{equation*}
for $j=0,\ldots,n-1$ and $s=1,\ldots,m$.

To show the Theorem~\ref{Thm:PGST}~(\ref{Cond:PGST2}) holds, consider integers $l_{j,s}$'s satisfying
\begin{equation}
\label{Eqn:MPGSTrooted1}
\sum_{j=0}^{n-1} \sum_{s=1}^m l_{j,s} \lambda_{j,s}=0.
\end{equation}
We apply the trace  from $M$ to $\QQ(\gamma)$ to both sides.
Applying Theorem~\ref{Thm:Trace}~(\ref{Prop:Trace3}) and Equation~(\ref{Eqn:TraceM}), 
Equation~(\ref{Eqn:MPGSTrooted1}) implies
\begin{equation*}
\sum_{j=0}^{n-1} \sum_{s=1}^m l_{j,s} (\gamma + \theta_j) = 
\gamma \left(\sum_{j=0}^{n-1} \sum_{s=1}^m l_{j,s}\right) + \sum_{j=0}^{n-1} \theta_j \left(\sum_{s=1}^m l_{j,s} \right)
=0. 
\end{equation*}
Since $\gamma$ is transcendental and $\sum_j \theta_j \left(\sum_s l_{j,s} \right)\in \QQ$, Equation~(\ref{Eqn:MPGSTrooted1}) is equivalent to
\begin{equation}
\label{Eqn:MPGSTrooted2a}
\sum_{j=0}^{n-1} \sum_{s=1}^m l_{j,s}=0
\end{equation}
and
\begin{equation}
\label{Eqn:MPGSTrooted2b}
\sum_{j=0}^{n-1} \theta_j \left(\sum_{s=1}^m l_{j,s} \right)=0.
\end{equation}
Recall $\theta_j = \alpha+\beta(jh+c_jn)$ where $\gcd(h,n)=1$.  Equations~(\ref{Eqn:MPGSTrooted2a})
and (\ref{Eqn:MPGSTrooted2b}) imply
\begin{equation*}
\sum_{j=0}^{n-1} (jh+c_jn) \left(\sum_{s=1}^m l_{j,s} \right)=0.
\end{equation*}
Since $\gcd(h,n)=1$, we have
\begin{equation*}
\sum_{j=0}^{n-1} j\sum_{s=1}^m l_{j,s}  = 0 \pmod{n}.
\end{equation*}
If Equations~(\ref{Eqn:MPGSTrooted2a})
and (\ref{Eqn:MPGSTrooted2b}) hold then, for any $\delta \in \RR$,
\begin{equation*}
\sum_{j=0}^{n-1}\sum_{s=1}^m l_{j,s}\left(q_{j,s}\left((x_0,h), (x_1,h)\right) + \delta\right)
=  \frac{2\pi}{n}\left(\sum_{j=0}^{n-1} j\sum_{s=1}^m l_{j,s} \right) + \delta \left(\sum_{j=0}^{n-1}\sum_{s=1}^m l_{j,s}\right)
= 0 \pmod{2\pi}.
\end{equation*}
By Theorem~\ref{Thm:PGST}, pretty good state transfer occurs from $(x_0,h)$ to $(x_1,h)$, for $h=1,\ldots,m$.
\end{proof}

\begin{remark}
\ 
\begin{itemize}
\item
Putting a transcendental weight $\gamma$ on the loops is sufficient for $\varphi_{0,m}(t), \ldots, \varphi_{n-1,m}(t)$ to be irreducible over $\QQ(\gamma)$.  
Theorem~5.8 holds for irrational number $\gamma$ as long as $\varphi_{0,m}(t), \ldots, \varphi_{n-1,m}(t)$ are irreducible over $\QQ(\gamma)$.
\item
If we move the loops from the $(x_a,1)$ to $(x_a,m)$, for $a=0,\ldots,n-1$, then a similar argument shows that
the resulting graph has multiple pretty good state transfer on the set $\{(x_0,h), (x_1,h), \ldots, (x_{n-1},h)\}$, for $h=1,\ldots,m$.
\end{itemize}
\end{remark}
\section*{Acknowledgements}

This project was completed under the 2021 Fields Undergraduate Summer Research Program which provided support for A. Acuaviva, S. Eldridge, M. How and E. Wright.  
C. Godsil gratefully acknowledges the support of the Natural Sciences and Engineering Council of Canada (NSERC)
Grant No. RGPIN-9439.  A. Chan is grateful for the support of the NSERC Grant No. RGPIN-2021-03609.
\printbibliography

\end{document}